\documentclass[12pt]{amsart}
\usepackage{amsmath,amsthm,amssymb,bbm,dsfont}
\usepackage{graphicx}
\usepackage{subfig}
\usepackage{cite,float}

\numberwithin{equation}{section}
\newtheorem{theorem}{Theorem}[section]
\newtheorem{lemma}[theorem]{Lemma}
\newtheorem{result}[theorem]{Result}
\newtheorem{proposition}[theorem]{Proposition}
\newtheorem{corollary}[theorem]{Corollary}
\theoremstyle{remark}
\newtheorem{remark}[theorem]{Remark}
\newtheorem{example}[theorem]{Example}
\newtheorem{definition}[theorem]{Definition}

\makeatletter
\@namedef{subjclassname@2010}{%
  \textup{2010} Mathematics Subject Classification}
\makeatother
\begin{document}

\title{Squeezing function corresponding to polydisk}

\author[N. Gupta]{Naveen Gupta}
\address{Department of Mathematics, University of Delhi,
Delhi--110 007, India}
\email{ssguptanaveen@gmail.com}

\author[S. Kumar]{Sanjay Kumar Pant}

\address{Department of Mathematics, Deen Dayal Upadhyaya college, University of Delhi,
Delhi--110 078, India}
\email{skpant@ddu.du.ac.in}

\begin{abstract}
In the present article,  we define squeezing function corresponding to polydisk
 and study its properties. We investigate relationship between squeezing fuction  and 
 squeezing function corresponding to polydisk. We also give an alternate proof for lower bound
 of the squeezing function of a product domain.
 \end{abstract}
\maketitle

\section{Introduction}

The work in this article is motivated by a question posed by Forn\ae ss in his talk \cite{talk}.
Forn\ae ss posed the question {\em ``what is analogous theory of squeezing function when embeddings 
	are taken into polydisk instead of unit ball?"} This question seems to 
rely on the fact that
unit ball and unit polydisk in $\mathbb{C}^n$, $n>1$ are not biholomorphic\cite{rudin}. 
Therefore
unit polydisk are most suitable objects to see what happens to the squeezing function 
when we replace
unit ball by it. 
In this paper, we consider this problem. We define squeezing function in 
the case when embeddings are taken into unit polydisk\,---\,in the definition of squeezing
function\,---\,instead of unit ball and we explore its various properties.
It seems to us in the present work that the results connected with squeezing 
function corresponding to polydisk
do not depend in a big way on the nature of polydisk even it not being biholomorphic to 
unit ball. 
\smallskip

The notion of squeezing function started with the work of Liu et al. \cite{Yau2004}, 
\cite{Yau2005}, in  2004 and 2005, in which they studied holomorphic homogeneous 
regular(HHR) manifolds. In 2009, Yeung in his paper \cite{yeung}, studied this notion. 
He renamed the HHR property as uniform squeezing property. The formal definition
of squeezing function\,---\,motivated by the work of Liu et al. \cite{Yau2004},
\cite{Yau2005} and Yeung \cite{yeung}\,---\,was introduced by Deng et al. 
\cite{2012}. In their paper \cite{2012},
the authors presented several properties of squeezing functions. The squeezing 
function has been the focus of great interest in the last few years. Indeed, several
new directions in the study of the geometrical and analytical aspects of the 
squeezing function have recently been introduced:
\cite{uniform-squeezing}, \cite{kim-joo}, \cite{andreev}, \cite{zimmer2}, 
\cite{kaushalprachi}, \cite{recent}. 
\smallskip

Let us denote by $\mathbb{B}^n$ unit ball in $\mathbb{C}^n$ and let $\Omega\subseteq
\mathbb{C}^n$ be a bounded domain.
For $z\in{\Omega}$ and a holomorphic embedding $f:\Omega\to \mathbb{B}^n$
with $f(z)=0$,  let $$S_{\Omega}(f,z):=\sup \{r:\mathbb{B}^n(0,r)\subseteq f(\Omega)\},$$ 
where $\mathbb{B}^n(0,r)$ denotes
ball of radius $r$, centered at origin.
Squeezing function on $\Omega$, denoted by $S_{\Omega}$, is defined as
$$S_{\Omega}(z):=\sup_f\{S_{\Omega}(f,z)\},$$ 
where supremum is taken over holomorphic embeddings $f:\Omega\to \mathbb{B}^n$ 
with $f(z)=0.$
\smallskip

The fact that squeezing function is biholomorphic invariant follows from the definition. 
A bounded domain is called holomorphic homogeneous regular if its squeezing function
has a positive lower bound. 
\smallskip

The definition of squeezing function corresponding to polydisk can be written 
exactly in similar manner as the definition of squeezing function by replacing 
unit ball with unit polydisk.
\begin{definition}
	Let us denote by $\mathbb{D}^n$ unit polydisk in $\mathbb{C}^n$ and let 
	$\Omega\subseteq \mathbb{C}^n$ be a bounded domain.
	For $z\in{\Omega}$ and a holomorphic embedding $f:\Omega\to \mathbb{D}^n$
	with $f(z)=0$, let $$T_{\Omega}(f,z):=\sup 
	\{r:\mathbb{D}^n(0,r)\subseteq f(\Omega)\},$$ 
	where  $\mathbb{D}^n(0,r)$ denotes polydisk of radius $r$, centered at origin.
	Squeezing function corresponding to polydisk on $\Omega$, denoted by $T_{\Omega}$,
	is defined as
	$$T_{\Omega}(z):=\sup_f\{T_{\Omega}(f,z)\},$$ 
	where supremum is taken over holomorphic embeddings $f:\Omega\to \mathbb{D}^n$ 
	with $f(z)=0.$
\end{definition}
In what follows, by squeezing function
we mean squeezing function corresponding to polydisk.
Relation between the squeezing function(denoted by $T_{\Omega}$) 
and the squeezing function corresponding
to the unit ball(denoted by $S_{\Omega}$), for any bounded domain $\Omega$,
is given by the following lemma:

\begin{lemma}\label{relatinglemma}
	Let $\Omega\subset \mathbb{C}^n$ be a bounded domain.
	Then for every $z\in{\Omega}$,
	\begin{enumerate} 
		\item[a)] $T_{\Omega}(z)\geq \dfrac{1}{\sqrt{n}}S_{\Omega}(z);$ 
		\item[b)] $S_{\Omega}(z)\geq \dfrac{1}{\sqrt{n}}T_{\Omega}(z).$
	\end{enumerate}
\end{lemma}
We say, a domain $\Omega$ is holomorphic homogeneous regular domain, 
if its squeezing function has a positive lower bound.
It follows from Lemma \ref{relatinglemma} that if 
one of them has a positive lower bound then the other has it too. 
Therefore holomorphic homogeneous regular domains are the same, 
in both the settings.

For a domain $\Omega\subseteq \mathbb{C}^n$, Fridman invariant(corresponding 
to unit polydisk), denoted by $h_{\Omega}^c$ is defined as
$$h_{\Omega}^c(a):=\sup\{\tanh r :B_{\Omega}^c(a,r)\subseteq f(\mathbb{D}^n), f:\mathbb{D}^n
\to \Omega \, \mbox{is holomorphic embedding}\},$$
where $c_{\Omega}$ is the Carathéodory pseudodistance and $B_{\Omega}^c(a,r)$ is
Carathéodory  ball centered at $a$ of radius $r>0.$ We allow $h^c_{\Omega}(a)=0$. 
Note that in \cite{frid1979}, \cite{frid1983} $\inf \frac 1 r$ is considered
in place of $\sup \tanh r.$

\section{Basic Properties}
We start with the simple observation that $T_{\Omega}$, for any bounded domain 
$\Omega$ is biholomorphic invariant.
To see this, let $f_0:\Omega_1\to \Omega_2$ be a biholomorphism and $z\in \Omega_1$. 
We claim that $T_{\Omega_1}(z)=T_{\Omega_2}(f_0(z))$.
\smallskip
\begin{itemize}
	
	\item	Let $f:\Omega_1\to \mathbb{D}^n$ be a holomorphic embedding with 
	$f(z)=0$. Take $h=f\circ g:\Omega_2\to \mathbb{D}^n,$ 
	where $g=f_0^{-1}:\Omega_2\to \Omega_1$. Then $h$ is holomorphic embedding
	with $h(f_0(z))=0$. Also note that $h(\Omega_2)=f(\Omega_1)$. 
	Let $r>0$ be such that $\mathbb{D}^n(0,r)\subseteq f(\Omega_1)=h(\Omega_2)$, 
	which gives us that $T_{\Omega_2}(f_0(z))\geq r$. Thus 
	\begin{equation*} 
		T_{\Omega_2}(f_0(z))\geq T_{\Omega_1}(z).
	\end{equation*}
	
	\item	On similar lines, we can prove that $T_{\Omega_1}(z)\geq T_{\Omega_2}(f_0(z))$,
	which will give us
	$$T_{\Omega_1}(z)= T_{\Omega_2}(f_0(z)).$$ 
\end{itemize}
\begin{proof}[Proof of Lemma \ref{relatinglemma}] Let $z\in {\Omega}$.
	\begin{enumerate} 
		\item[a)] Consider a holomorphic embedding $f:\Omega\to \mathbb{B}^n$ 
		such that $f(z)=0$. Let $r>0$ be
		such that $\mathbb{B}^n(0,r)\subseteq f(\Omega)\subseteq \mathbb{B}^n$. 
		Then by considering $f$ as a holomorphic
		embedding $f:\Omega\to \mathbb{D}^n$ and noticing that $\mathbb{D}^n
		\left(0,\dfrac{r}{\sqrt{n}}\right)
		\subseteq \mathbb{B}^n(0,r)$, 
		we get part a) of the lemma.
		\item[b)] Consider a holomorphic embedding $f:\Omega\to \mathbb{D}^n$ 
		such that $f(z)=0$. Let $r>0$ be such that
		$\mathbb{D}^n(0,r)\subseteq f(\Omega)\subseteq \mathbb{D}^n$. 
		Let $g:\Omega\to \mathbb{B}^n$ 
		be defined as $g(z)=\dfrac{1}{\sqrt{n}}f(z)$.
		then 
		$$\mathbb{B}^n\left(0,\dfrac{r}{\sqrt{n}}\right)\subseteq \mathbb{D}^n
		\left(0,\dfrac{r}{\sqrt{n}}\right)\subseteq g(\Omega)\subseteq \mathbb{B}^n,$$
		which proves part b) of the lemma. 
	\end{enumerate}
\end{proof}

Note that this lemma gives us how the two squeezing functions $T_{\Omega}$ and $S_{\Omega}$ 
are related and that the holomorphic homogeneous 
regular domains turn out to be the same in both the settings.
\begin{remark}\label{remark:relatinglemma}
	The two inequalities in Lemma \ref{relatinglemma} can not be attained simultaneously. One can see
	that it can not happen unless $T_{\Omega}(z)=0=S_{\Omega}(z)$.
\end{remark}
The fact that for any bounded domain $\Omega$, $T_{\Omega}$  is continuous and the existence of
extremal map for $T_\Omega$ follows on the same lines as in \cite{2012}, with some minor changes
in the arguments.
For the sake of completeness, we give the proof here. We first establish existence of extremal
property for $T_\Omega$. We will need the following theorem for this.

\begin{result}[{\cite[Theorem~2.2]{2012}}]\label{injectivity} 
	Let $\Omega\subseteq \mathbb{C}^n$ be a bounded domain and $z\in{\Omega}$. Let $\{f_i\}$
	be a sequence of injective holomorphic
	maps, $f_i:\Omega\to \mathbb{C}^n$, with $f_i(z)=0$ for all $i.$
	Suppose that $f_i\to f,$ uniformly on compact subsets of $\Omega$,
	where $f:\Omega\to \mathbb{C}^n$. If there exists a neighborhood
	$U$ of $0$ such that $U\subseteq f_i(\Omega)$ for all $i$, then $f$ is 
	injective.
\end{result}

The following theorem establishes existence of extremal maps for squeezing 
function.
\begin{theorem} \label{extremal}
	Let $\Omega\subseteq \mathbb{C}^n$ be a bounded domain. Let $z\in \Omega$, 
	then there exists a holomorphic embedding $f:\Omega\to \mathbb{D}^n$ with $f(z)=0$
	such that $\mathbb{D}^n(0,T_{\Omega}(z))\subseteq f(\Omega).$
\end{theorem}
\begin{proof}
	Let $z\in \Omega$ and  $r=T_{\Omega}(z)$. Then by definition, 
	there exists an increasing sequence of numbers $r_i$, converging to $r$, 
	a sequence of holomorphic embeddings $f_i:\Omega\to \mathbb{D}^n$ with 
	$f_i(z)=0$ such that $$\mathbb{D}^n(0,r_i)\subseteq f_i(\Omega)\subseteq \mathbb{D}^n.$$
	
	Thus $\{f_i\}$ is locally bounded and therefore it is normal by
	Montel's theorem. Let $\{f_{i_k}\}$ be a
	subsequence of $f_i$ such that $f_{i_k}\to f$; 
	where $f:\Omega\to \mathbb{C}^n$ is holomorphic.
	\smallskip
	
	Note that $f(z)=0$. Also, since $\mathbb{D}^n(0,r_2)\subseteq f_2(\Omega)$, 
	therefore $$\mathbb{D}^n(0,r_1)\subseteq \mathbb{D}^n(0,r_2)
	\subseteq f_2(\Omega).$$
	
	Similarly, it is easy to check that
	$\mathbb{D}^n(0,r_1)\subseteq f_i(\Omega)$ for all $i$, which gives us injectivity of $f$ 
	by using Result $\ref{injectivity}.$
	Note that, since $f$ is open map, we get $f:\Omega\to \mathbb{D}^n.$ 
	Finally we need to show that 
	$\mathbb{D}^n(0,T_{\Omega}(z))\subseteq f(\Omega).$ To prove this, it is sufficient to prove
	$\mathbb{D}^n(0,r_j)\subseteq f(\Omega)$
	for every fixed $j$. 
	\smallskip
	
	Since $r_i$ is an increasing sequence, we have 
	$\mathbb{D}^n(0,r_j)\subseteq f_i(\Omega)$ for all $i>j.$
	Let $g_i=f_i^{-1}|_{\mathbb{D}^n(0,r_j)}$,
	then $f_{i_k}\circ g_{i_k}=\mathbbm{Id}_{\mathbb{D}^n(0,r_j)} $ for $i_k>j$. 
	Without loss of generality, let us denote by $g_{i_k}$, a subsequence of $g_{i_k} $, 
	which exists by Montel's theorem, converging 
	to a function $g:\mathbb{D}^n(0,r_j)\to \mathbb{C}^n$, uniformly on compact subsets of 
	$\mathbb{D}^n(0,r_j)$. 
	\smallskip
	
	Clearly, $g:\mathbb{D}^n(0,r_j)\to \overline{\Omega}$. We claim that 
	$g:\mathbb{D}^n(0,r_j)\to \Omega$. 
	For this, first note that $g$ is defined on some neighborhood of the closure
	$\overline{\mathbb{D}^n(0,r_j)}$. Let $\zeta 
	\in{g(\mathbb{D}^n(0,r_j)})\setminus \Omega$. Let $\tilde {g_{i_k}}(z)=
	g_{i_k}(z)-\zeta$ and $\tilde {g}(z)=g(z)-\zeta$ for $z\in \mathbb{D}^n(0,r_j) .$
	Since $g_{i_k}\left(\mathbb{D}^n(0,r_j)\right)\subseteq  \Omega$, therefore $\tilde{g_{i_k}}$ has no
	zero in $\mathbb{D}^n(0,r_j)$ and $\tilde{g}$ has a zero in $\mathbb{D}^n(0,r_j)$.
	Let $z_0\in \mathbb{D}^n(0,r_j)$ be such that $g(z_0)=\zeta$, that is, $\tilde{g}(z_0)=0$.
	Since $g$ is locally biholomorphism, there is some $\delta>0$ such that $z_0$ is
	the unique zero of $\tilde{g}$ on $\overline{\mathbb{B}^n(z_0,\delta)}.$ Take $\epsilon 
	=\inf \{|\tilde{g}(z)|:\partial \mathbb{B}^n(z_0,\delta)\}$ and note that $\epsilon >0$.
	Now using convergence of $\tilde{g_{i_k}}$ for this $\epsilon$ and then using 
	\cite[~Theorem 3]{Rouche}, we get that $\tilde{g_{i_k}}$ has a zero in $\mathbb{B}^n(z_0,\delta)$
	for sufficiently large $k$. It is a contradiction therefore 
	$g\left( \mathbb{D}^n(0,r_j)\right)\subseteq \Omega$ for each $j$.
	
	\smallskip
	
	This gives us that $f\circ g:\mathbb{D}^n(0,r_j)\to \mathbb{D}^n(0,r_j)$ 
	is well defined and $f\circ g=\mathbbm{Id}_{\mathbb{D}^n(0,r_j)}$. Therefore we get that 
	$\mathbb{D}^n(0,r_j)\subseteq f(\Omega)$
	and hence the result is obtained. 
	
\end{proof}
\begin{remark}\label{rem:extremal}
	From Theorem \ref{extremal}, it is easy to see that if $T_{\Omega}(z)=1$ for some 
	$z\in{\Omega}$, then $\Omega$ is biholomorphic to unit polydisk.
\end{remark}

Let us denote by $K_{\Omega}(.,.)$, Kobayashi distance on 
any domain $\Omega\subseteq \mathbb{C}^n$ and 
let us define $\sigma:[0,1)\to R^+ $ as $\sigma(x)=\log \dfrac{1+x}{1-x}.$ 
It is easy to check that
$\sigma$ is one-one with inverse given by $\sigma^{-1}(y)=\tanh \left(\dfrac y 2\right)$. 
Now we are ready to prove that $T_\Omega$ for any bounded domain $\Omega$ is continuous.
\begin{theorem}
	For any bounded domain $\Omega\subseteq \mathbb{C}^n$, the function $T_{\Omega}$ 
	is continuous.
\end{theorem}
\begin{proof}
	Let $a\in{\Omega}$ be arbitrarily fixed. Let $\{z^i\}$ be a sequence in $\Omega$ 
	such that $z^i\to a$.
	We will prove that $T_{\Omega}(z^k)\to T_{\Omega}(a)$.
	\smallskip
	
	Let $\epsilon >0$ be arbitrary. For $a\in \Omega$, there exist $f:\Omega\to \mathbb{D}^n $ 
	such that $f(a)=0$ with $\mathbb{D}^n(0,T_{\Omega}(a))\subseteq f(\Omega).$ 
	Since $f$ is continuous, there exists $N$ such that $\|f(z^k)\|<\epsilon$ for $k>N$.
	For $k>N$, consider $f^k:\Omega\to \mathbb{C}^n$ 
	defined as $$f^k(z)=\dfrac{f(z)-f(z^k)}{1+\epsilon}.$$
	Note that $f^k(\Omega)\subseteq \mathbb{D}^n$ and $f^k(z^k)=0$. We first prove that 
	$\mathbb{D}^n\left(0,\dfrac{T_{\Omega}(a)-\epsilon}{1+\epsilon}\right)\subseteq
	f^k(\Omega).$ 
	\smallskip
	
	Let $z=(z_1,\ldots,z_n)\in{\mathbb{D}^n\left(0,\dfrac{T_{\Omega}(a)-\epsilon}
		{1+\epsilon}\right)}$,
	then $|z_i|<\dfrac{T_{\Omega}(a)-\epsilon}{1+\epsilon}$ for all $i=1,\ldots,n.$
	Consider  $|z_i(1+\epsilon)+f_i(z^k)|\leq |z_i|(1+\epsilon)+\epsilon<T_{\Omega}(a),$ 
	for all $i=1,2,\ldots,n$ and therefore, we get $z(1+\epsilon)+f(z^k)\in{f(\Omega)}$. 
	So $z\in{f^k(\Omega)}$, 
	which gives us $T_{\Omega}(z^k)\geq \dfrac{T_{\Omega}(a)-\epsilon}{1+\epsilon}$ 
	for all $k>N$. Thus 
	\begin{equation}\label{eqn:ex1}
		\liminf_{k\to \infty} T_{\Omega}(z^k)\geq T_{\Omega}(a).
	\end{equation}
	On the other hand, notice that 
	$K_{\Omega}(z^k,a)\to K_{\Omega}(a,a)$ as $z^k\to a.$ Also for each $k$, 
	there exists holomorphic 
	embedding $f^k:\Omega\to \mathbb{D}^n$ with $f^k(z^k)=0$ such that 
	$\mathbb{D}^n(0,T_{\Omega}(z^k))\subseteq f^k(\Omega).$
	\smallskip
	
	Now, $K_{\mathbb{D}^n}(f^k(z^k),f^k(a))=K_{\mathbb{D}^n}(0,f^k(a))\leq 
	K_{\Omega}(z^k,a)\to 0$, 
	as $k\to \infty$. Thus $f^k(a)\to 0$, in usual as we know that Kobayashi metric
	induces standard topology\cite{kobayashi-usual}.
	\smallskip
	
	Let $\epsilon >0$ be arbitrary. Since $f^k(a)\to 0$, therefore there exists $N$ 
	such that $\|f^k(a)\|<\epsilon$ for all $k>N$. 
	For $k>N$, define $f:\Omega\to \mathbb{C}^n$ as 
	$$f(z)=\dfrac{f^k(z)-f^k(a)}{1+\epsilon},$$ then $f(a)=0$. 
	We claim that $f(\Omega)\subseteq \mathbb{D}^n.$
	\smallskip
	
	Let $f=(f_1,f_2,\ldots, f_n)$, then it is easy to verify that
	$|f_j(z)|<1$ for all $j=1,2,\ldots,n$, 
	which gives us $f(\Omega)\subseteq \mathbb{D}^n.$ 
	Also, as argued in the previous case it is easy to check that 
	$\mathbb{D}^n\left(0,\dfrac{T_{\Omega}(z^k)-\epsilon}{1+\epsilon}\right)\subseteq(f(\Omega)).$
	Therefore we get $T_{\Omega}(a)\geq \dfrac{T_{\Omega}(z^k)-\epsilon}{1+\epsilon}$, 
	which further gives us that 
	\begin{equation}\label{eqn:ex2}
		T_{\Omega}(a)\geq \limsup_{k\to \infty} T_{\Omega}(z^k).
	\end{equation} 
	
	Hence the result follows from Equation \ref{eqn:ex1} and \ref{eqn:ex2}.
\end{proof}
\begin{theorem}\label{extension}
	Let $\Omega'\subseteq \mathbb{C}^n$ be a bounded domain and let $A\subseteq \Omega'$ 
	be proper analytic subset. Then for $\Omega=\Omega'\setminus A,$ 
	$$T_{\Omega}(z)\leq \sigma^{-1}\left(K_{\Omega'}(z,A)\right),\ z\in \Omega.$$
	In particular, $\Omega$ is not homogeneous regular.
\end{theorem}
\begin{proof}
	Let $z\in{\Omega},$ and  $f:\Omega\to \mathbb{D}^n$ be a holomorphic embedding 
	such that $f(z)=0.$
	Since $\Omega=\Omega'\setminus A$, therefore by Riemann removable singularity theorem, $f$ 
	can be extended to a holomorphic map $\tilde{f}:\Omega'\to \mathbb{D}^n$.
	Clearly, $\tilde{f}(\Omega)\cap \tilde{f}(A)=\emptyset$.
	We know by decreasing property of Kobayashi metric that 
	$K_{\mathbb{D}^n}(\tilde{f}(z),\tilde{f}(A))\leq K_{\Omega'}(z,A)$, 
	that is $K_{\mathbb{D}^n}(0,\tilde{f}(A))\leq K_{\Omega'}(z,A)$.
	
	Let $r>0$ be such that $\mathbb{D}^n(0,r)\subseteq f(\Omega)=\tilde{f}(\Omega)$. 
	Thus, $\mathbb{D}^n(0,r)\cap \tilde{f}(A)=\emptyset$. 
	This implies that for each $a\in A, \ |\tilde{f_i}(a)|\geq r$ for some $1\leq i\leq n,$ 
	which further gives us  $$K_{\mathbb{D}^n}(0,(r,0,\ldots,0))\leq K_{\mathbb{D}^n}(0,\tilde{f}(A))
	\leq K_{\Omega'}(z,A),$$
	which implies that $r\leq \sigma^{-1}(K_{\Omega'}(z,A))$. 
	Hence we get that $$T_{\Omega}(z)\leq \sigma^{-1}\left(K_{\Omega'}(z,A)\right).$$
\end{proof}
\section{Examples}
In this section, we discuss squeezing function for some domains. 
We start with unit polydisk and unit ball.
\begin{example}\label{ex: 1}
	Let $\Omega_1=\mathbb{D}^n $ and $\Omega_2=\mathbb{B}^n$. 
	\begin{itemize}
		\item Considering automorphisms of $\mathbb{D}^n$, 
		it is easy to see that $T_{\Omega_1}(z)\equiv 1$. 
		\item Note that, using Theorem {\ref{extremal}},  and the fact that
		$\mathbb{B}^n$ and $\mathbb{D}^n\, , n>1$ are not biholomorphic, we get that $T_{\Omega_2}(z)<1$. 
		Also by Lemma \ref{relatinglemma}, $T_{\Omega_2}(z)\geq \frac{1}{\sqrt{n}}S_{\Omega_2}(z)$. 
		Since we know that $S_{\Omega_2}(z)\equiv 1$, we get 
		$\frac{1}{\sqrt{n}}\leq T_{\Omega_2}(z)\leq 1$.
		Also by \cite[Proposition~2]{alexander}, we get that 
		$T_{\Omega_2}(z)\leq \frac{1}{\sqrt{n}}$. 
		Thus, $T_{\Omega_2}(z)\equiv \frac{1}{\sqrt{n}}.$
	\end{itemize}
\end{example}
\begin{example}\label{ex: ball}
	Let us take $\Omega_1'=\mathbb{D}^n$, $\Omega_1=\Omega_1'\setminus \{0\}$,
	$\Omega_2'=\mathbb{B}^n$  and $\Omega_2=\Omega_2'\setminus \{0\}$. 
	\begin{itemize}
		\item For any $z\in{\Omega}$, using Theorem \ref{extension}, we get that 
		$T_{\Omega}(z)\leq \displaystyle{\max_{1\leq i\leq n}\|z_i\|}$. Also by considering automorphism $f$ of $\mathbb{D}^n,$
		with $f(z)=0$, it is obvious that $\displaystyle{\mathbb{D}^n(0,\max_{1\leq i\leq n}\|z_i\|)\subseteq f(\Omega_1)}$ which further 
		gives us $T_{\Omega}(z)\geq \displaystyle{\max_{1\leq i\leq n}\|z_i\|}$. Thus we obtain
		$T_{\Omega_1}(z)= \displaystyle{\max_{1\leq i\leq n}\|z_i\|}.$ 
		\item Using Theorem \ref{extension}, we get that $T_{\Omega_2}(z)\leq \|z\|.$
		\smallskip
		
		For $z\in{\Omega_2}$, consider a holomorphic embedding $f:\Omega_2\to \mathbb{D}^n$ with $f(z)=0$.
		Then $f$ can be extended to a holomorphic embedding $F:\Omega_{2}'\to \mathbb{D}^n$. 
		\smallskip
		
		Let $r>0$ be such that $\mathbb{D}^n(0,r)\subseteq f(\Omega_2)\subseteq \mathbb{D}^n$, and therefore
		$\mathbb{D}^n(0,r)\subseteq F(\Omega_2')\subseteq \mathbb{D}^n$,
		therefore by 
		{\cite[Proposition~2]{alexander}} ,
		we get that 
		\begin{equation}\label{eqn}
			T_{\Omega_2}(z)\leq \dfrac{1}{\sqrt{n}}\ \mbox{for all}\ z\in{\Omega_2}.
		\end{equation}
		Let $z\in{\Omega_2}$ be such that $\|z\|<\frac{1}{\sqrt n}$. 
		Consider an automorphism $\phi_z$ of $\Omega_2'$, 
		which maps $z$ to $0$ and $0$ to $z$. Thus $f=\phi_z|_{\Omega_2}:\Omega_2\to \mathbb{D}^n$ is a 
		holomorphic embedding with $f(z)=0$. 
		It is easy to see that $\displaystyle{\mathbb{D}^n(0,\max_{1\leq i\leq n}\|z\|)}
		\subseteq f(\Omega_2)= \Omega_2'\setminus \{z\}.$ 
		So we get that $T_{\Omega_2}(z)\geq \displaystyle{\max_{1\leq i\leq n}\|z\|}$. 
		We also have $T_{\Omega_2}(z)\leq \|z\|$ by Theorem \ref{extension} and
		therefore $\displaystyle{\max_{1\leq i\leq n}\|z\|\leq T_{\Omega_2}(z)\leq \|z\|}$ 
		on $\mathbb{B}^n\left(0,\frac{1}{\sqrt{n}}\right)\setminus \{0\}$ and hence on 
		$\overline{\mathbb{B}^n\left(0,\frac{1}{\sqrt{n}}\right)}$ by continuity of $T_{\Omega_2}$.
		\smallskip
		
		Also, for $\|z\|>\frac{1}{\sqrt{n}}$, by Lemma \ref{relatinglemma} and Equation \ref{eqn}, 
		we have that $\frac{\|z\|}{\sqrt{n}}\leq T_{\Omega_2}(z)\leq \frac{1}{\sqrt{n}}.$ 
		Here we have used that $S_{\Omega_2}(z)=\|z\|,$ by  {\cite[Corollary~7.3]{2012}}.
	\end{itemize}
\end{example}
\section{Classical Symmetric Domains}

We now study squeezing function for classical symmetric bounded domains. 
A classical domain is one of the following four types:
\smallskip

$R_1(r,s)=\{Z:I-Z\overline{Z}'>0, Z\ \mbox{is an}\ r\times s\ \mbox{matrix}\}\ (r\leq s),$

$R_{II}(p)=\{Z:I-Z\overline{Z}'>0, Z\ \mbox{is a symmetric matrix of order}\ p\},$

$R_{III}(q)=\{Z:I-Z\overline{Z}'>0, Z\ \mbox{is a  skew symmetric matrix of order}\ q \},$

$R_{IV}(n)=\{z=(z_1,z_2,\ldots, z_n)\in{\mathbb{C}^n}:1+|zz'|^2-2zz'>0, \ 1-|zz'|>0\}$.
\smallskip

Note that for a bounded homogeneous domain $H$, its squeezing function is constant, 
using biholomorphic invariance of squeezing functions. Let us denote this constant by $t(H).$

\begin{theorem}\label{classical}
	$$\frac{1}{\sqrt{n}\sqrt{r}}\leq t\left(R_I(r,s)\right)\leq \sqrt{\frac{n}{r}},$$
	$$\frac{1}{\sqrt{n}\sqrt{p}}\leq t\left(R_{II}(p)\right)\leq \sqrt{\frac{n}{p}},$$
	$$\frac{1}{\sqrt{n} \sqrt{\left[\frac q 2 \right]}} \leq t \left(R_{III}(q)\right)\leq 
	\sqrt{\frac{n}{\left[\frac q 2 \right]}},$$
	$$\frac{1}{\sqrt{n} \sqrt{2}} \leq t \left(R_{IV}(n)\right)\leq \sqrt{\frac{n}{2}}.$$
	
\end{theorem}

\begin{theorem}\label{classical product}
	If $R_1,R_2,\ldots, R_k$ are classical symmetric domains, and 
	$D=R_1\times R_2\times \ldots \times R_k$ then 
	$$s(D) \frac{1}{\sqrt{n}}\leq t(D)\leq s(D), $$
	where $n=n_1+n_2+\ldots+n_k$, $n_i$ is complex dimension of $R_i,\ i=1,2,\ldots,k$ and $s(D)$ is,
	squeezing function(which is constant since domain is homogeneous) 
	of $D$\,---\,corresponding to unit disc\,---\,given 
	by $$s(D)=\left[s(R_1)^{-2}+s(R_2)^{-2}+\ldots+ s(R_k)^{-2}\right]^{-\frac{1}{2}}.$$
\end{theorem}

We will need the following key lemma, whose proof is based on the method of Kubota \cite{kubota}.
\begin{lemma}\label{classical lemma}
	Let $\Omega$ be a bounded homogeneous domain in $\mathbb{C}^n$ satisfying the conditions:
	\begin{enumerate}
		\item $	\{z = (z_1,\ldots,z_n): |z_{\alpha_j}|<1, \  \mbox{for}\ j = 1,\ldots,m\ 
		\mbox{and}\  z_{\alpha} = 0\  \mbox{for the other}\  \alpha 's\} \subseteq \Omega$,
		where $1 \leq \alpha_1\leq   \ldots \leq a_m\leq n;$
		\item for each $(1 \leq j \leq m)$,
		$$\{z = (z_1,\ldots ,z_n): |z_{\alpha_j}| = 1 \ \mbox{and}\  z_{\alpha} = 0\ 
		\mbox{for the other}\  \alpha 's\} \subseteq \partial \Omega.$$
		If $f:\Omega\to \mathbb{D}^n$ is holomorphic embedding and if $\mathbb{D}^n(0,R)
		\subseteq f(\Omega),$ then $R \leq \sqrt{\frac{n}{m}} $.
	\end{enumerate}
\end{lemma}
\begin{proof}
	We have $f=(f_1,f_2,\ldots, f_n),$ where each $f_i:\Omega \to \mathbb{D},$ 
	where $\mathbb{D}$ is unit disk in $\mathbb{C}.$ Since $\mathbb{D}^n(0,R)\subseteq
	f(\Omega)$,
	therefore we have that 
	\begin{equation}\label{liminf}
		R \leq  \liminf_{r\uparrow 1} \max_{1\leq i\leq n} |f_i(0,\ldots, re^{i\theta_j},\ldots, 0)|.
	\end{equation}
	
	Define a function $g:\mathbb{D}^m\to \mathbb{D}^n$ by $g=(g_1,g_2,\ldots, g_n),$
	where $$g_i(\zeta_1,\zeta_2,\ldots, \zeta_m)=
	f_i(0,\ldots,\zeta_1,\zeta_2,\ldots, \zeta_m,\ldots, 0 ),\ i=1,2,\ldots, n;$$
	where $\zeta_j$ is in $\alpha_j $th position in right hand side expression.
	Let us write taylor series expansion of $g_{i}$ as $g_i(\zeta)=\sum a_v^{(i)}
	\zeta^v$, where $v$ is multi index and $a_0^{(i)}=0.$
	
	By  condition (1) and using Equation \ref{liminf} we get
	\begin{equation}\label{parseval}
		R^2\leq \frac{1}{2\pi}\int_0^{2\pi} \|g(0,\ldots, e^{i\theta_{j}},\ldots, 0)\|^2,
	\end{equation}
	where $g(0,\ldots, e^{i\theta_{j}},\ldots, 0)$ is defined as $\lim_{r\to 1}g(0,\ldots,
	re^{i\theta_{j}},\ldots, 0)$ 
	and $e^{i \theta_j}$ lies in the $\alpha_j $th position. Notice that the expression on the
	right side of the inequaltiy \ref{parseval} equals
	$\sum_{i=1}^{n}\sum |a_{0\ldots v_j\ldots 0}^{(i) }|^2$. 
	Thus summing it over all positions $\alpha_1,\alpha_2,\ldots, \alpha_m,$ we get 
	\begin{equation}\notag
		\begin{split}
			mR^2&\leq \sum_{i=1}^{n}\sum |a_{0\ldots v_1\ldots 0}^{(i) }|^2+
			\sum_{i=1}^{n}\sum |a_{0\ldots v_2\ldots 0}^{(i) }|^2+\ldots+\sum_{i=1}^{n}\sum |a_{0\ldots v_m\ldots 0}^{(i) }|^2\\
			&\leq \sum_{i=1}^{n}\sum |a_{v_1v_2\ldots v_m}^{(\alpha) }|^2\\
			&\leq  \int_{\mathbb{T}^m}\|g\|^2 dh\\
			&\leq  n.
		\end{split}
	\end{equation} 
	Here we have used that $\|g\|^2\leq n$ a.e. on $\mathbb{T}^m$.
	Hence we get our result.
\end{proof}
Now the proof of Theorem \ref{classical} and Theorem \ref{classical product}\,---\,using the method 
of Kubota \cite{kubota}\,---\,follows from  Lemma \ref{relatinglemma}
and Lemma \ref{classical lemma}. 

Recently T.W. Ng, C.C. Tang and J. Tsai  gave the result \cite[Theorem~2]{recent} on lower
bound of  squeezing function corresponding to unit ball for the  product of planar domains. As remarked by them, 
by modifying the argument of their proof one can obtain the following more general result.
\begin{result}\label{product domain s}
	Let $\Omega_i\subseteq \mathbb{C}^{n_i},\ i=1,2,\ldots, k$ be bounded domains 
	and $\Omega=\Omega_1\times \Omega_2\times \ldots\times \Omega_k\subseteq 
	\mathbb{C}^n,\ n=n_1+n_1+\ldots+n_k$. 
	Then for any $a=(a_1,a_2,\ldots, a_k)\in \Omega,
	\ a_i\in{\Omega_i},\ i=1,2,\ldots, k$ , we have $$S_{\Omega}(a)\geq \left[(S_{\Omega_1}(a_1))^{-2}+
	(S_{\Omega_2}(a_2))^{-2}+\ldots+(S_{\Omega_k}(a_k))^{-2}\right]^{-\frac{1}{2}}.$$
\end{result}

In \cite{kubota}, Kubota proved that equality in Proposition \ref{product domain s} is obtained,
when the product of classical symmetric domains is considered.  That is 
\begin{equation}\label{eqn:equality}
	S_{\Omega}(a)= \left[(S_{\Omega_1}(a_1))^{-2}+
	(S_{\Omega_2}(a_2))^{-2}+\ldots+(S_{\Omega_k}(a_k))^{-2}\right]^{-\frac{1}{2}}.
\end{equation}
We give here the following example to show that such an equality does not hold in case of
squeezing function.

\begin{example}
	Let us consider $\Omega_1=\mathbb{D}=\Omega_2\subseteq \mathbb{C}$.
	\smallskip
	
	Let $\Omega=\Omega_1\times \Omega_2\subseteq \mathbb{C}^2$ and 
	$a=(a_1,a_2)\in \Omega,\ a_i\in{\Omega_i},\ i=1,2.$ 
	Then clearly, $T_{\Omega_1}\equiv 1\equiv T_{\Omega_2}$. 
	Thus we get $\left[(T_{\Omega_1}(a_1))^{-2}+
	(T_{\Omega_2}(a_2))^{-2}\right]^{-\frac{1}{2}}=\frac{1}{\sqrt{2}}.$ Also 
	$T_\Omega(a)\equiv 1$, 
	therefore we get
	$$T_\Omega(a)=1 >\frac{1}{\sqrt{2}}.$$ Thus we see that Equation 
	\ref{eqn:equality}	does not hold.
\end{example} 

We give the following result which gives us lower bound for squeezing function 
of product domain.
\begin{proposition}\label{product domain t}
	Let $\Omega_i\subseteq \mathbb{C}^{n_i},\ i=1,2,\ldots, k$ be bounded domains 
	and $\Omega=\Omega_1\times \Omega_2\times \ldots \times \Omega_k\subseteq 
	\mathbb{C}^n,\ n=n_1+n_1+\ldots+n_k$. 
	Then for any $a=(a_1,a_2,\ldots, a_k)\in \Omega,
	\ a_i\in{\Omega_i},\ i=1,2,\ldots, k$, we have $$T_{\Omega}(a)\geq 
	\min_{1\leq i\leq n} T_{\Omega_i}(a_i).$$
	In particular, product of holomorphic homogeneous regular domains is holomorphic
	homogeneous regular.
\end{proposition}
\begin{proof}
	By Theorem \ref{extremal}, for each $a_i\in{\Omega_{n_i}}$, 
	there is a holomorphic embedding $f_i:\Omega_i\to \mathbb{D}^{n_i}$,
	with $f_i(a_i)=0$ such that 
	\begin{equation}\label{eqn:product}
		\mathbb{D}^{n_i}(0,T_{\Omega_{i}}(a_i))\subseteq f_i
		(\Omega_i)\subseteq \mathbb{D}^{n_i},\ i=1,2,\ldots, k.
	\end{equation}
	Consider the map $f:\Omega \to \mathbb{D}^n$ defined as
	$$f(z_1,z_2,\ldots, z_k)=\left(f_1(z_1),f_2(z_2),\ldots,f_k(z_k)\right).$$
	Clearly, $f$ is holomorphic embedding with $f(a)=0$. Now let us take 
	$w=(w_1,w_2,\ldots, w_k)
	\in \mathbb{D}^{n_1}(0,r)\times \mathbb{D}^{n_2}(0,r)\times \ldots
	\times \mathbb{D}^{n_k}(0,r), $ where $r$ is given by $\displaystyle{r= \min_{1\leq i\leq n} T_{\Omega_i}(a_i)}$. By
	Equation \ref{eqn:product}, there exists $b_i\in{\Omega_i},$ such that $f_i(b_i)=w_i,\ i=1,2,
	\ldots, k,$ since
	$\mathbb{D}^{n_i}(0,r)\subseteq \mathbb{D}^{n_i}(0,T_{\Omega_i}(a_i)) $ for each $i$.
	Thus $w=f(b_1,b_2,\ldots, b_k)$ and as $w$ was arbitrarily chosen, we conclude 
	$\mathbb{D}^{n}(0,r)\subseteq f(\Omega)$. 
	Thus it follows from the definition that $T_{\Omega}(a)\geq 
	\displaystyle{\min_{1\leq i\leq n} T_{\Omega_i}(a_i)}.$
\end{proof}
\begin{remark}
	The above  succinct proof suggested by Gautam Bharali is an obvious replacement  of the proof given by us in earlier version  {\underline{arXiv:2007.14363v1}} of the article.
\end{remark}
\begin{remark}
	The conclusion that the product of holomorphic homogeneous regular domains is
	holomorphic homogeneous regular follows directly from Result \ref{product domain s} 
	and Lemma \ref{relatinglemma} too. We have given this proof to point that
	the lower bound  obtained in 
	Proposition \ref{product domain t} for squeezing function is an improvement over 
	the lower bound obtained in Result \ref{product domain s} for squeezing
	fuction corresponding to unit ball. This is ensured by the following simple
	observation.
\end{remark}

\begin{lemma}
	Let $0<\alpha_i\leq 1,\ i=1,2,\ldots, k$, then for each $i$
	\begin{equation*}\label{eqn:compare}
		\alpha_i\geq \left(\alpha_1^{-2}+\alpha_2^{-2}+\ldots+\alpha_k^{-2}\right)^{-\frac{1}{2}}.
	\end{equation*}
\end{lemma}
Now our assertion about improvement of lower bound follows by taking 
$\alpha_i=T_{\Omega_i}(a_i),\ i=1,2,\ldots, k.$

\section{Stability of squeezing function}
In \cite{2016}, authors discussed the relation between limit of squeezing function (corresponding 
to unit ball) of a sequence of domains and the squeezing function (corresponding 
to unit ball) of the limit domain. We investigate this relation for squeezing function in this section. 

\begin{theorem}\label{increasing-stab}
	Let $\Omega\subseteq \mathbb{C}^n$ be a bounded domain and let $\Omega_k\subseteq \Omega,
	\ k\in{\mathbb{N}}$ be such that $\Omega_k\subseteq \Omega_{k+1}$ for all $k$ and
	$\Omega=\cup_k \Omega_k$. Then 
	for any $z\in{\Omega}$, $$\lim_{k\to\infty} T_{\Omega_k}(z)=T_{\Omega}(z).$$
\end{theorem}
\begin{proof}
	Let $z\in{\Omega}$, then $z\in{\Omega_{k_0}}$ for some $k_0$ and therefore $z\in 
	\Omega_k$ for $k>k_0.$ For each $k>k_0$, 
	let $f_k:\Omega_k \to \mathbb{D}^n$ be a holomorphic embedding with $f_k(z)=0$ 
	such that $\mathbb{D}^n(0,T_{\Omega_k}(z))\subseteq f_k(D)$ (Theorem \ref{extremal}).
	Montel's theorem ensures that $\{f_k\}$ has a subsequence which converges to a function 
	$f:\Omega\to \mathbb{C}^n$
	on compact subsets of $\Omega$. Since each $f_k:\Omega_k \to \mathbb{D}^n$, it is
	obvious that $f:\Omega \to \overline{\mathbb{D}^n}$. We claim that 
	$f:\Omega \to \mathbb{D}^n$, for this we will prove that $f$ is injective.
	
	For $k>k_0$, since $S_{\Omega_k}(z)\geq \dfrac{d(z,\partial \Omega_k)}
	{diam \Omega_k}\geq \dfrac{d(z,\partial \Omega_{k_0})}{diam \Omega}$,
	therefore using Lemma \ref{relatinglemma} we get $$T_{\Omega_k}(z)\geq \frac{1}
	{\sqrt{n}}S_{\Omega_k}(z)\geq 
	\frac{1}{\sqrt{n}} \dfrac{d(z,\partial \Omega_k)} {diam \Omega_k}\geq
	\frac{1}{\sqrt{n}} \dfrac{d(z,\partial \Omega_{k_0})}{diam \Omega}.$$
	Thus there exists $\delta>0$ such that $\mathbb{D}^n(0,\delta)\subseteq f_k(\Omega_k)$
	for all $k>k_0$. Therefore for each $k>k_0,$ $g_k=f_k^{-1}|_{\mathbb{D}^n(0,\delta)}:
	\mathbb{D}^n(0,\delta)\to \Omega$ is well defined. 
	For $k>k_0$, considering Cauchy's inequality for $g_k$, we get that 
	$\left|det(dg_k(0))\right|< c,$ for some $c>0$. Now since $f_k\circ g_k=
	\mathbbm{Id}_{\mathbb{D}^n(0,\delta)}$, we get that $\left|det(df_k(z))\right|>\frac 1 c$ 
	and therefore $\left|det(df(z))\right|\geq \frac 1 c>0$. 
	Therefore $f$ is locally one-one at $z$ and is thus open using 
	\cite[Lemma~2.3]{2012} and \cite[Theorem~3]{Rouche}. Thus $f:\Omega\to \mathbb{D}^n.$
	
	In order to establish the theorem, we first prove $T_{\Omega}(z)\geq \limsup_k 
	T_{\Omega_k}(z)$. Let $\limsup_k T_{\Omega_k}(z)=r$ and let $\limsup_k 
	T_{\Omega_{k_i}}(z)$ be a subsequence such that 
	$$\lim_{k\to\infty}\limsup_k T_{\Omega_{k_i}}(z)=r.$$ It is easy to observe that $r>0$.
	Let $0<\epsilon<r,$ considering $r-\epsilon>0$ we get that $\mathbb{D}^n(0,r-\epsilon)
	\subseteq f_{k_i}(\Omega_{k_i})$ for large $k_i$. Let $h_{k_i}=
	f_{k_i}^{-1}|_{\mathbb{D}^n(0,r-\epsilon)}:\mathbb{D}^n(0,r-\epsilon)\to \Omega_{k_i}
	\subseteq \Omega.$ By Montel's theorem, we may assume(by relabeling the indices) that $h_{k_i}$ 
	converges to $h:\mathbb{D}^n(0,r-\epsilon)\to \mathbb{C}^n$ uniformly on compact subsets of
	$\mathbb{D}^n(0,r-\epsilon).$ Since each $h_{k_i}:\mathbb{D}^n(0,r-\epsilon)\to \Omega$, 
	therefore $h:\mathbb{D}^n(0,r-\epsilon)\to \overline{\Omega}.$ Now 
	$$\left|det(dh(0))\right|=\lim_i \left|det(dh_{k_i}(0))\right|=\lim_i 
	\left|det(df_{k_i}^{-1}(0))\right|=\left|det(df^{-1}(0))\right| \neq 0.$$
	This implies that $h$ is injective and hence open.  So it follows that $h:\mathbb{D}^n
	(0,r-\epsilon)\to \Omega$. Thus $f\circ h$ is well defined and it is obvious that 
	$f\circ h=\mathbbm{Id}_{\mathbb{D}^n(0,r-\epsilon)}.$
	This gives us $\mathbb{D}^n(0,r-\epsilon)\subseteq f(\Omega)$, which further implies that
	$T_{\Omega}(z)\geq r-\epsilon$. Therefore we get $T_{\Omega}(z)\geq \limsup_k 
	T_{\Omega_k}(z).$ 
	
	Now we prove that $T_{\Omega}(z)\leq \liminf_k T_{\Omega_k}(z)$. Consider subsequence
	$T_{\Omega_{k_i'}}(z)$ converging to $\liminf_k T_{\Omega_k}(z).$ Using Theorem \ref{extremal}
	there exists holomorphic embedding $g:\Omega \to \mathbb{D}^n$ with $g(z)=0$ such that
	$\mathbb{D}^n(0,T_{\Omega}(z))\subseteq g(\Omega).$ Then for arbitrary $\epsilon>0$ with
	$0<\epsilon <T_{\Omega}(z)$, we have $g^{-1}\left(\mathbb{D}^n(0,T_{\Omega}(z)-\epsilon)\right)
	\subseteq \Omega_{k_i'}$ for large $k_i'.$ This implies that 
	$$\mathbb{D}^n(0,T_{\Omega}(z)-\epsilon)\subseteq g(\Omega_{k_i'})$$
	for large $k_i'$. Therefore we get $T_{\Omega}(z)-\epsilon \leq  T_{\Omega_{k_i'}}(z)$. 
	This further gives us that $$T_{\Omega}(z)-\epsilon \leq \lim_i T_{\Omega_{k_i'}}(z)=
	\liminf_k T_{\Omega_k}(z).$$
	Hence the result follows since $\epsilon$ is arbitrary.
\end{proof}
As pointed out in \cite[Lemma~2.2]{deng2019}, the proof of the following corollary follows from
the proof of Theorem \ref{increasing-stab}.
\begin{corollary}\label{lem:increasing-stab}
	Let $\Omega,\ \Omega_k\subseteq \mathbb{C}^n(k\geq 1)$ be bounded domains, where
	$\Omega_k\subseteq \Omega$ with the condition that for each compact subset $K\subseteq
	\Omega$ there exists $N\in{\mathbb{N}}$ such that $K\subseteq \Omega_k$ for all $k\geq N$.
	Then for any $z\in{\Omega}$, $$\lim_{k\to\infty} T_{\Omega_k}(z)=T_{\Omega}(z).$$
\end{corollary}
\begin{proof}
	Let $z\in{\Omega},$ then using hypothesis for $K=\{z\},$ there exists $N\in{\mathbb{N}}$
	such that $z\in{\Omega_k}$ for each $k>N.$ Now the proof follows proceeding as in
	Theorem \ref{increasing-stab}.
\end{proof}
Similar to the case of squeezing function corresponding to unit ball, we have the following theorem for 
a decreasing sequence of domains. We omit the details of the proof since it 
follows by the same arguments as in 
\cite[Theorem~2.2]{2016}, modifying the arguments for polydisk as we did in the proof of Theorem
\ref{increasing-stab}.
\begin{theorem}\label{decreasing-stab}
	Let $\Omega\subseteq \mathbb{C}^n$ be a bounded domain and let $\Omega\subseteq \Omega_k,
	\ k\in{\mathbb{N}}$ be such that $\Omega_k\supseteq \Omega_{k+1}$ for all $k$ and
	$\Omega=\cap_k \Omega_k$. Then 
	for any $z\in{\Omega}$, $$\limsup_k T_{\Omega_k}(z)\leq T_{\Omega}(z).$$
\end{theorem}
Using Lemma \ref{relatinglemma} and \cite[Example~2.1]{2016}, we get that strict inequality 
may occur in Theorem \ref{decreasing-stab}.

In \cite[Proposition~1]{kaushal}, authors gave relation between squeezing function
corresponding to unit ball and the corresponding Fridman invariant. We prove 
the same relation for squeezing function and the corresponding Fridman invariant
in the following proposition. 
\begin{proposition}\label{fridrelating}
	Let $\Omega\subseteq \mathbb{C}^n$ be a bounded domain, then for any $a\in{\Omega}$, 
	$$T_{\Omega}(a)\leq h_{\Omega}^c(a).$$
\end{proposition}
\begin{proof}
	For $a\in{\Omega}$,
	let $f:\Omega\to \mathbb{D}^n$ be a holomorphic embedding such that $f(a)=0.$ Let $r>0$ 
	be such that $\mathbb{D}^n(0,r)\subseteq f(\Omega).$ Consider $g:\mathbb{D}^n\to\Omega$
	defined as $g(z)=f^{-1}(rz).$ We claim that $B^c_{\Omega}\left(a,\tanh^{-1}r\right)
	\subseteq g(\mathbb{D}^n) \subseteq \Omega.$ Let $w\in{B^c_{\Omega}\left(a,\tanh^{-1}r\right)},$
	then 
	\begin{align*}
		\tanh^{-1}r&>c_{\Omega}(a,w)\\
		&=c_{f(\Omega)}(f(a),f(w))\\
		&=c_{f(\Omega)}(0,f(w))\\
		&\geq c_{\mathbb{D}^n}(0,f(w))\\
		&=\max_{1\leq i\leq n}\rho(0,f_i(w)),
	\end{align*} where $\rho$ denotes Poincar\' e metric on unit disk in $\mathbb{C}.$
	This implies that $|f_i(w)|<r$ for all $i=1,2,\ldots, n$, that is,
	$f(w)\in{\mathbb{D}^n(0,r)}.$ Thus we get $w\in{f^{-1}(\mathbb{D}^n(0,r))}=g({\mathbb{D}^n})$
	and this further implies our claim. Therefore $r\leq h_{\Omega}^c(a)$ and hence we 
	get $T_{\Omega}(a)\leq h_{\Omega}^c(a)$.
\end{proof}

\begin{remark}\label{rem-fridrelating}
	Let $\Omega= \mathbb{D}^n$, then $T_{\Omega}\equiv 1$. Also  $h_{\Omega}^c\equiv 1$
	(see \cite{frid1979} for proof)
	therefore for unit polydisk we have $T_{\Omega}(a)=h_{\Omega}^c(a)$ for every $a\in{\Omega}$.
	Thus using biholomorphic invariance of $T_\Omega$ and $h_{\Omega}^c$, we get that $T_{\Omega}(a)
	= h_{\Omega}^c(a)$, whenever $\Omega$ is biholomorphic to unit polydisk. 
	But equality can also be obtained for domains,
	which are not biholomprhic to unit polydisk. For example take $\Omega=B^n$, then
	$T_{\Omega}\equiv h_{\Omega}^c$. 
	
\end{remark}

\begin{remark}
	For $\Omega_1=\mathbb{B}^n$ and $\Omega_2=\mathbb{D}^n,$ we have 
	$T_{\Omega_1}\equiv \frac{1}{\sqrt{n}}$,
	$T_{\Omega_2}\equiv 1$, $S_{\Omega_1}\equiv 1$ and $S_{\Omega_2}\equiv \frac{1}{\sqrt{n}}$. 
	Therefore $T_{\Omega_1}= \frac{1}{\sqrt{n}} S_{\Omega_1}$ and 
	$S_{\Omega_2}= \frac{1}{\sqrt{n}} T_{\Omega_2}.$ Also for classical domain $R_1(r,s),$ 
	if we take $r=1,$ then $R_1(1,s)=\mathbb{B}^s$, and therefore
	$T_{R_1(1,s)}\equiv \frac{1}{\sqrt{s}}\equiv \frac{1}{\sqrt{s}}S_{R_1(1,s)}$.  
	In Example \ref{ex: ball}, we noticed that such an equality does not hold
	for the punctured unit ball. We strongly feel that for a bounded homogeneous 
	domain $\Omega$, equality for one of the two inequalities in Lemma \ref{relatinglemma} holds. 
\end{remark}

\medskip
\section*{Acknowledgement}
	We thank Gautam Bharali for his critical and valuable insights for improving the article. 
	We profusely thank the referee for several valuable comments and suggestions.

\end{document}